\documentclass[11pt, reqno]{amsart}

\usepackage{thmtools}

\usepackage[utf8]{inputenc}
\usepackage{amsfonts,amssymb, amsmath}
\usepackage{mathrsfs}
\usepackage{tikz-cd}
\usepackage{enumitem}

\usepackage{graphicx}
\usepackage{color}
\usepackage{hyperref}

\usetikzlibrary{arrows}

\tikzset{
    nat/.style  = {-implies,double, double equal sign distance}
}

\theoremstyle{plain}
\newtheorem{theorem}{Theorem}[section]
\newtheorem{proposition}[theorem]{Proposition}
\newtheorem{lemma}[theorem]{Lemma}

\newtheorem{corollary}[theorem]{Corollary}

\newtheorem*{claim*}{Claim}

\theoremstyle{definition}

\newtheorem{example}[theorem]{Example}
\newtheorem{remark}[theorem]{Remark}

\newenvironment{tikzcdnat}[1][]{\begin{tikzcd}[arrows=nat,#1]}{\end{tikzcd}}

\newcommand\one{\mathbf 1}
\newcommand\A{\mathcal A}
\newcommand\B{\mathcal B}
\newcommand\C{\mathbb C}
\newcommand\D{\mathbb D}

\newcommand\DC{\mathbb D}   %
\newcommand\Id{\mathrm{Id}}
\newcommand\id{\mathrm{id}}
\renewcommand\epsilon{\varepsilon}
\newcommand\fl{^\dagger} %

\newcommand\colim{\mathop{\mathrm{colim}}}
\newcommand\ol{\overline}

\newcommand\Graph{\mathbf{Graph}}
\newcommand\GraphFin{\mathbf{Graph}_{\mathsf{fin}}}
\newcommand\Cli{\mathbf{Cli}}   %
\newcommand\DClip{\DC_{\Cli_p}}  %
\newcommand\Rs{\mathcal{R}(\sigma)}
\newcommand\R{\mathbb R}
\newcommand\EM{\mathtt{EM}}
\newcommand\Ek{\mathbb{E}_k}
\newcommand\Pk{\mathbb{P}_k}
\newcommand\PRk{\mathbb{PR}_k}

\newcommand\nattrans{\Rightarrow}

\newcommand\qtq[1]{\quad\text{#1}\quad}

\newcommand{\ef}{Ehrenfeucht--Fra{\"i}ss{\'e}}
\newcommand{\compbased}{com\-po\-nent-based}
\newcommand{\decomposable}{componental}

\newenvironment{prf}{\begin{proof}}{\end{proof}}

\newcommand\appendixmentioned[1]{}

\begin{document}

\title{Discrete density comonads and graph parameters}

\thanks{
    Accepted at CMCS 2022. The final publication will be available at Springer. \\
    \indent
    \ Supported by EPSRC grant EP/T00696X/1: Resources and Coresources: a junction between categorical semantics and descriptive complexity}

\author{Samson Abramsky} %
\address{Department of Computer Science, University College London}
\email{s.abramsky@ucl.ac.uk}
\author{Tom\'a\v s Jakl} %
\address{Department of Computer Science and Technology, University of Cambridge}
\email{tomas.jakl@cl.cam.ac.uk}
\urladdr{https://tomas.jakl.one}
\author{Thomas Paine}
\address{Department of Computer Science, University of Oxford}
\email{thomas.paine@stcatz.ox.ac.uk}

\maketitle

\begin{abstract}
Game comonads have brought forth a new approach to studying finite model theory
categorically. By representing model comparison games semantically as comonads,
they allow important logical and combinatorial properties to be exressed in terms
of their Eilenberg-Moore coalgebras. As a result, a number of  results
from finite model theory, such as preservation theorems and homomorphism counting theorems, have been
formalised and parameterised by comonads, giving rise to new results simply by 
varying the comonad.

In this paper we study the limits of the comonadic approach in the
combinatorial and homomorphism-counting aspect of the theory, regardless of whether any
model comparison games are involved. We show that any standard graph
parameter has a corresponding comonad, classifying the same class. This comonad
is constructed via a simple Kan extension formula, making it the initial
solution to this problem and, furthermore, automatically admitting a
homomorphism-counting theorem.
\end{abstract}

\section{Introduction}

An important feature of the emerging theory of game comonads \cite{abramskydawarwang2017pebbling,abramskyshah2021relating,conghailedawar2021game,abramskymarsden2021comonadic} is that game comonads classify a number of important classes of finite relational structures. We say that a comonad $\C$ \emph{classifies a class} $\Delta$ if a finite structure $A$ is in the class $\Delta$ precisely when $A$ admits a $\C$-coalgebra. For example, the Ehrenfeucht–Fra\"iss\'e comonad $\Ek$ classifies the structures of tree-depth ${\leq}k$ and, similarly, the pebbling comonad $\Pk$ classifies tree-width ${<}k$.

In this paper we study the theoretical limits of the comonadic approach. In particular, we aim to identify classes of structures which can be classified by comonads. We can readily predict two necessary requirements. Since the problem is stated in the language of category theory, we know that the classes of structures classified by comonads need to be closed under isomorphisms and, moreover, since finite coalgebras are closed under binary coproducts $+$ (i.e.\ disjoint unions), this must also be the case for the classes classifiable by comonads.

In fact, we show that  one further, very natural requirement suffices in order to be able to classify a class of structures by a comonad. It suffices to assume that the class is closed under connected substructures. 

We define a class $\Delta$ of finite relational structures or graphs to be \emph{\compbased{}} if it is closed under 
\begin{itemize}
\item isomorphisms,
\item finite coproducts, and
\item summands (i.e. if $A + B$ is in $\Delta$ then so are $A$ and $B$).
\end{itemize}

We can now state our first main result.
\begin{theorem}\label{t:class-classification}
    Any \emph{\compbased{}} class $\Delta$ can be classified by a comonad.
\end{theorem}

The theorem applies to a wide variety of classes of structures studied in the literature. In particular, these assumptions hold for all classes of finite structures classified by our game comonads and, moreover, for a number of typical examples of classes of structures for which a given graph parameter is bounded by a constant. For example, we obtain comonads for planar graphs, bipartite graphs, or graphs of  degree or clique-width bounded by a constant. Moreover, we show that the constructed comonad $\C$ is weakly initial among the comonads classifying $\Delta$, meaning that for any comonad $\D$ classifying $\Delta$, there is a comonad morphism $\C \nattrans \D$. This initiality allows us to obtain a characterisation of comonads that classify monotone \emph{nowhere dense classes} \cite{nevsetvril2011nowhere}.

Another important aspect of game comonads is that they classify various well-known binary relations between relational structures. We say that a comonad $\C$ \emph{classifies relation} $\asymp$ whenever $A \asymp B$ holds precisely when the cofree $\C$-coalgebras on $A$ and $B$ are isomorphic. For example, the comonad $\Ek$ classifies the relation that expresses that Duplicator has a winning strategy in the bijective $k$-round variant of the Ehrenfeucht–Fra\"iss\'e game \cite{abramskyshah2021relating} and, similarly,  $\Pk$ classifies the existence of a winning strategy in the bijective $k$-pebble game \cite{abramskydawarwang2017pebbling}. Furthermore, it was recently shown that the relation classified by $\Ek$ admits a Lov\'asz-type theorem. In particular, finite structures $A,B$ have isomorphic cofree $\Ek$-coalgebras if, and only if, they admit the same homomorphism counts from finite structures of tree-depth ${\leq}\, k$, i.e.\ when there is a bijection between $\hom(C,A)$ and $\hom(C,B)$ for every finite $C$ of tree-depth ${\leq}\, k$. Similar Lov\'asz-type theorems have  also been shown for $\Pk$ and the pebble-relation $\PRk$ comonads \cite{dawarjaklreggio2021lov,montacuteshah2021pebble}.

We show that the comonad constructed in the proof of Theorem~\ref{t:class-classification} automatically admits a Lov\'asz-type theorem for the class of structures it classifies. In fact, we show that such a comonad always has \emph{finite rank}\footnote{Comonads of finite rank should not be confused with finitary comonads, which is a weaker notion.}, which ensures that the category of coalgebras for the comonad is locally finitely presentable (cf.\ \cite[Proposition 1.12.1]{diers1986categories}, see also \cite[Appendix B]{reggio2021polyadic}) and therefore, by a recent result of Luca Reggio \cite[Corollary 5.15]{reggio2021polyadic}, admits the following Lov\'asz-type result.

\begin{corollary}\label{c:relation-classification}
    Let $\Delta$ be a \compbased{} class of finite structures and let $\asymp$ be the binary relation on finite structures such that, for any two finite structures $A,B$,
    \[ A \asymp B \iff \hom(C,A) \cong \hom(C,B) \quad\text{for all } C\in \Delta \]
    The comonad classifying $\Delta$ by Theorem~\ref{t:class-classification} also classifies $\asymp$.
\end{corollary}

As an example, we obtain that the comonad obtained by Theorem~\ref{t:class-classification} which classifies planar graphs also classifies quantum isomorphism (cf.\ \cite{manvcinskaroberson2020quantum}), and similarly, the comonad for coproducts of cycles classifies co-spectrality, and the comonad for bipartite graphs classifies isomorphic bipartite double covers.

\bigskip
The density comonad construction is the main technical tool of this paper. In fact, we develop most of our theory by means of discrete density comonads, that is, density comonads of functors with discrete domain. A general overview of the necessary categorical terminology and results is given in Section~\ref{s:prelim}. Discrete density comonads are introduced in Section~\ref{s:discrete-dens-com} and Theorem~\ref{t:class-classification} is proved in Section~\ref{s:classif-theorem}. In Section~\ref{s:graph-param} we take a look at how graph parameters correspond to coalgebra numbers of graded comonads, compare discrete density comonads with game comonads, and characterise comonads classifying monotone nowhere dense classes of graphs. Lastly, in Section~\ref{s:lovasz} we prove Corollary~\ref{c:relation-classification} by showing that, under mild conditions, discrete density comonads have finite rank.

\section{Preliminaries}
\label{s:prelim}

In this section we fix notation and  recall some basic facts about comonads and the density construction.
We assume the reader is familiar with elementary category theory notions such as functors, natural transformations, adjunctions, limits and colimits (see e.g.~\cite{abramskytzevelekos2010introduction} or \cite{awodey2010category}).

Throughout the paper we use the following notation. Given a natural transformation $\lambda\colon E\Rightarrow F$ between functors $E,F\colon \A \to \B$ and functors $G\colon \B \to \B'$ and $H\colon \A' \to \A$, we denote by $G\lambda$ and $\lambda H$ the obvious natural transformations of type $GE \Rightarrow GF$ and $EH \Rightarrow FH$, respectively.

\subsection{Comonads and coalgebras}
A \emph{comonad} (on category $\A$) is a triple $(\C,\epsilon,\delta)$ where $\C\colon \A \to \A$ is an endofunctor, and $\epsilon\colon \C \nattrans \Id$ and $\delta\colon \C \nattrans \C^2$ are natural transformations such that the following diagrams commute.
\[
    \begin{tikzcdnat}
        \C \rar{\delta}\dar[swap]{\delta} & \C^2 \dar{\delta \C}  \\
        \C^2 \rar[swap]{\C\delta} & \C^3
    \end{tikzcdnat}
    \qquad
    \begin{tikzcdnat}
        \C \rar{\delta} \ar{dr}[swap]{\id} & \C^2 \dar{\epsilon \C} \\
        & \C
    \end{tikzcdnat}
    \qquad
    \begin{tikzcdnat}
        \C \rar{\delta} \ar{dr}[swap]{\id} & \C^2 \dar{\C\epsilon} \\
        & \C
    \end{tikzcdnat}
\]

A morphism $\alpha\colon A \to \C(A)$ is an \emph{(Eilenberg--Moore) $\C$-coalgebra}\footnotemark{} if the following diagrams commute.
\begin{equation}
    \begin{tikzcd}
        A \dar[swap]{\alpha} \ar{dr}{\id} \\
        \C(A) \rar[swap]{\epsilon_A} & A
    \end{tikzcd}
    \qquad
    \begin{tikzcd}
        A \rar{\alpha}\dar[swap]{\alpha} & \C^2(A) \dar{\delta_A}  \\
        \C(A) \rar[swap]{\C\alpha} & \C^2(A)
    \end{tikzcd}
    \label{e:coalg}
\end{equation}
We say that $A$ \emph{admits a coalgebra} if there exists a morphism $A\to \C(A)$ which is a $\C$-coalgebra.
\footnotetext{All coalgebras in this text are Eilenberg--Moore coalgebras. We do not work with functor coalgebras at any point.}

Coalgebras form a category $\EM(\C)$ where morphisms between coalgebras $(A,\alpha) \to (B,\beta)$ are morphisms $h\colon A \to B$ such that $\beta \circ h = \C(h) \circ \alpha$. Moreover, there is a pair of adjoint functors
\[  U^\C\colon \EM(\C) \to \A \qtq{and} F^\C\colon \A \to \EM(\C) \]
between $\EM(\C)$ and the underlying category $\A$.
The left adjoint is just a forgetful functor, it sends a coalgebra $\alpha\colon A\to \C(A)$ to its underlining object $U^\C(A,\alpha) = A$. The right adjoint returns the \emph{cofree coalgebra $F^\C(A)$ on $A$}, represented by the morphism $\delta_A\colon \C(A)\to \C^2(A)$.

\subsection{Comonad morphisms}
\label{s:comonad-morph}
Given two comonads $(\C,\epsilon^\C, \delta^\C)$ and $(\D,\epsilon^\D,\delta^\D)$ on $\A$, a natural transformation $\lambda\colon \C \nattrans \D$ is a \emph{comonad morphism} if the following two diagrams of natural transformations commute.
\begin{align}
    \begin{tikzcdnat}[ampersand replacement=\&]
    \C \ar{rr}{\lambda}\ar{rd}[swap]{\epsilon^\C} \& \& \D \ar{ld}{\epsilon^\C}\\
    \& \Id{}
    \end{tikzcdnat}
    \qquad
    \qquad
    \begin{tikzcdnat}[ampersand replacement=\&]
    \C \ar[swap]{d}{\delta^\C}\ar{rr}{\lambda} \& \& \D \ar{d}{\delta^\D} \\
    \C^2 \rar{\C\lambda} \& \C\D \rar{\lambda\D} \& \D^2
    \end{tikzcdnat}
    \label{e:comonad-morphisms}
\end{align}

Note that comonad morphisms can be equivalently presented as functors $L\colon \EM(\C) \to \EM(\D)$ such that the following diagram of functors commutes.
\[
\begin{tikzcd}
\EM(\C) \ar[swap]{rd}{U^\C} \ar{rr}{L} & & \EM(\D) \ar{ld}{U^\D} \\
& \A
\end{tikzcd}
\]
The functor $L$ is constructed from the comonad morphism given as a natural transformation $\lambda$ by sending $A\xrightarrow{\alpha} \C(A)$ to $A\xrightarrow{\alpha} \C(A) \xrightarrow{\lambda_A} \D(A)$. For details, see e.g.\ \cite{wisbauer2008algebras}.

\subsection{Density comonads}\label{s:density-comonads}

Next, we review basics of the theory of density comonads, introduced independently by Appelgate and Tierney~\cite{appelgate1969categories} and Kock~\cite{kock1966continuous} (who studied the dual notion of codensity monads). The \emph{density comonad} of a functor $M\colon \A \to \B$ is a functor $\DC_M\colon \B \to \B$ with a natural transformation $\eta\colon M \nattrans \DC_M M$.
\[ \begin{tikzcd}[sep=1.9em]
    \A
    \ar[rr, "M", ""{name=F, below}]
    \ar[swap]{dr}{M}
    & {} & \B \\
    &
    |[alias=C]| \B
    \ar[swap]{ur}{\DC_M}
    \ar[Rightarrow, from=F, to=C, "\eta",shorten >=0.4em,shorten <= 0.2em]
\end{tikzcd} \]
Moreover, $\eta$ is required to be the initial natural transformation with this property. In other words, for any functor $K\colon \B \to \B$ and a natural transformation $\varphi\colon M \nattrans KM$ there is a \emph{unique} $\varphi^*\colon \DC_M \nattrans K$ such that $\varphi = \varphi^*M \circ \eta$, i.e.\ diagramatically
\[
    \begin{tikzcd}[sep=1.9em]
        \A
        \ar[rr, "M", ""{name=F, below}]
        \ar[swap]{dr}{M}
        & {} & \B \\
        &
        |[alias=C]| \B
        \ar[swap]{ur}{K}
        \ar[Rightarrow, from=F, to=C, "\varphi",shorten >=0.4em,shorten <= 0.2em]
    \end{tikzcd}
    \quad = \quad
    \begin{tikzcd}[sep=2.8em]
        \A
        \ar[rr, "M", ""{name=F, below}]
        \ar[swap]{dr}{M}
        & {} & \B \\
        &
        |[alias=C]| \B
        \ar[ur, "\DC_M", pos=0.65, ""{name=Lan, below}]
        \ar[ur, swap, bend right=55, pos=0.55, "K", ""{name=K, above}]
        \ar[Rightarrow, from=F, to=C, "\eta",shorten >=0.4em,shorten <= 0.2em]
        \ar[Rightarrow, from=Lan, to=K, "\varphi^*", swap]
    \end{tikzcd}
\]
Density comonads are special types of left Kan extensions. They do not exist for all functors. However, when $\A$ is a small category and $\B$ is cocomplete then $\DC_M$ exists for every functor $M\colon \A \to \B$. In such case, $\DC_M(B)$ is computed as the colimit of the diagram:
\[ M \downarrow B \xrightarrow{V} \A \xrightarrow{M} \B \]
where $V$ is the forgetful functor from the comma category $M \downarrow B$, which consists of pairs $(A, f)$ where $f\colon M(A) \to B$ is a morphism in $\B$, and morphisms $(A, f) \to (A',f')$ between such pairs are morphisms $g\colon A \to A'$ in $\A$ making the following triangle commute.
\[
\begin{tikzcd}
    M(A) \ar[swap]{rd}{f}\ar{rr}{M(g)} & & M(A')\ar{ld}{f'} \\
    & B
\end{tikzcd}
\]
We may express the same fact by the formula:
\begin{align}
\DC_M(B) \enspace =
\colim_{
A\in \A,\  M(A) \to B
} M(A).
\label{e:density-comonad-formula}
\end{align}

Note that $\A$ does not have to be small nor $\B$ cocomplete in general. It is enough that the colimit above exists. In such case we speak of \emph{pointwise} density comonads. Denote by
\[ \iota_f\colon M(A) \to \DC_M(B) \]
the inclusion morphism of the copy of $M(A)$ corresponding to the morphism $f\colon M(A) \to B$ into the colimit. Then, the component $\eta_A\colon M(A) \to \DC_M(M(A))$ of the natural transformation $\eta\colon M \nattrans \DC_M M$ is given as $\iota_f$ for $f$ equal to the identity morphism $\id\colon M(A) \to M(A)$.

\subsection{The comonad structure}\label{s:lifting-M}
The initiality of $\eta\colon M \nattrans \DC_M M$ ensures that we can equip $\DC_M$ with a comonad structure. In particular, the identity natural transformation $M \nattrans \Id \circ M$ uniquely factors as the composition of $\eta$ with the counit $\epsilon\colon \DC_M \nattrans \Id$ and, similarly, $\DC_M(\eta) \circ \eta\colon M \nattrans \DC_M \circ \DC_M \circ M$ factors through the comultiplication ${\delta\colon \DC_M \nattrans \DC_M^2}$. In other words, the counit and the comultiplication are uniquely determined by the equations
\begin{align}
\epsilon M \circ \eta = \id \quad\text{and}\quad \delta M \circ \eta  = \DC_M(\eta) \circ \eta .
\label{eq:eta-comonad}
\end{align}

Moreover, these two equations guarantee that the functor
\begin{align}
    M\fl\colon \A \to \EM(\DC_M)
    \label{eq:M-fl}
\end{align}
which sends \(A\in \A\) to the coalgebra \(\eta_A\colon M(A) \to \DC_M(M(A))\) is well-defined.

Lastly, we recall three equations of density comonads, following from \eqref{eq:eta-comonad}, which we use extensively throughout the paper. For morphisms $f\colon M(A) \to B$ and $h\colon B\to C$, the following three triangles commute.
\[
\begin{matrix}
    \begin{tikzcd}[ampersand replacement=\&]
        M(A) \dar[swap]{\iota_f}\ar{dr}{\iota_{h\mkern1mu\circ f}} \\
        \DC_M(B) \rar[swap]{\DC_M(h)} \& \DC_M(C)
    \end{tikzcd}
\\[3.2em]
\text{(DC1)}
\end{matrix}
\quad
\begin{matrix}
\begin{tikzcd}[ampersand replacement=\&]
    M(A) \dar[swap]{\iota_f}\ar{dr}{f} \\
    \DC_M(B) \rar[swap]{\epsilon_B} \& B
\end{tikzcd}
\\[3.2em]
\text{(DC2)}
\end{matrix}
\quad
\begin{matrix}
\begin{tikzcd}[ampersand replacement=\&]
    M(A) \dar[swap]{\iota_f} \ar{dr}{\iota_{\iota_f}} \\
    \DC_M(B) \rar[swap]{\delta_B} \& \DC_M(\DC_M(B))
\end{tikzcd}
\\[3.2em]
\text{(DC3)}
\end{matrix}
\]

\subsection{Comonad morphisms from composites}
\label{s:composed-functors}

Let $M$ be the composite of functors
\[ \A_0 \xrightarrow{M_0} \A \xrightarrow{M_1} \B \]
such that the density comonads $\DC_M$ and $\DC_{M_1}$ exist. Let
\[ \eta\colon M \nattrans \DC_M M \qquad \eta^1\colon M_1 \nattrans \DC_{M_1} M_1 \]
be the corresponding initial natural transformations. By initiality of $\eta$, there is a unique natural transformation
\[ \lambda\colon \DC_M \nattrans \DC_{M_1} \]
such that:
\[
    \begin{tikzcd}[sep=1.9em]
        \A_0 \rar{M_0}
        & \A
        \ar[rr, "M_1", ""{name=F, below}]
        \ar[swap]{dr}{M_1}
        & {} & \B \\
        & &
        |[alias=C]| \B
        \ar[swap]{ur}{\DC_{M_1}}
        \ar[Rightarrow, from=F, to=C, "\eta^1",shorten >=0.4em,shorten <= 0.2em]
    \end{tikzcd}
    \quad = \quad
    \begin{tikzcd}[sep=3.7em]
        \A
        \ar[rr, "M", ""{name=F, below}]
        \ar[swap]{dr}{M}
        & {} & \B \\
        &
        |[alias=C]| \B
        \ar[ur, "\DC_M", pos=0.55, ""{name=Lan, below}]
        \ar[ur, swap,bend right=55, "\DC_{M_1}", ""{name=K, above}]
        \ar[Rightarrow, from=F, to=C, "\eta",shorten >=0.4em,shorten <= 0.2em]
        \ar[Rightarrow, from=Lan, to=K, "\lambda"]
    \end{tikzcd}
\]

\begin{lemma}\label{l:comonad-morphism}
$\lambda\colon \DC_M \Rightarrow \DC_{M_1}$ is a comonad morphism.
\end{lemma}

In fact, $\DC_{(-)}$ is a functor from the category of functors $X \to \B$ which admit density comonads %
into the category of comonads and comonad morphisms (cf.~page~73 in~\cite{dubuc2006kan}, see also~\cite{leinster2013codensity}).

\section{Discrete density comonads}
\label{s:discrete-dens-com}

Recall that a category is \emph{discrete} whenever it only has identity morphisms.
By \emph{discrete density comonads} we mean density comonads for functors whose domain is discrete. 
An important feature of discrete density comonads is that the formula in \eqref{e:density-comonad-formula} simplifies dramatically. Indeed, assume that
\[ M\colon \A \to \B \]
is a fixed functor from a small and discrete category $\A$. Then, since $\A$ is discrete, there are no morphisms between objects $M(A)$, given by morphisms ${M(A) \to B}$, in the colimit formula \eqref{e:density-comonad-formula}. Therefore, the density comonad $\DC_M\colon \B \to \B$ is computed as a coproduct, that is, the colimit of a discrete diagram. Concretely, for an object $B$ in $\B$,
\begin{align}
    \DC_M(B) \enspace=\enspace \coprod_{A\in \A} \ \coprod_{f\colon M(A)\to B} M(A).
    \label{e:discrete-formula}
\end{align}
Note that $\DC_M$ exists whenever the above coproduct exists in $\B$, for every object $B\in \B$. In particular, $\DC_M$ exists whenever $\B$ has all coproducts.

As with general density comonads, we have inclusion morphisms
\[ \iota_f\colon M(A) \to \DC_M(B) \]
for every $f\colon M(A) \to B$, which satisfy axioms (DC1)--(DC3) from Section~\ref{s:lifting-M}.

\medskip
For the proof of Theorem~\ref{t:class-classification}, the category $\B$ is either the category $\Rs$ of \mbox{$\sigma$-structures} (i.e.\ relational structures in a fixed relational signature $\sigma$) or the category $\Graph$ of graphs (where by graphs we mean undirected loopless graphs). The morphisms in these categories are the structure-preserving functions: $\sigma$-structure homomorphisms $f\colon A\to B$ satisfy that $R^A(x_1,\dots,x_n)$ implies $R^B(f(x_1),\dots,f(x_n))$ and, likewise, graph homomorphisms preserves the edge relation. For example, in the former case, we may describe the comonad~$\DC_M$ explicitly as follows. For a $\sigma$-structure~$B$, the universe of $\DC_M(B)$ consists of tuples
\[ (A, f, x) \]
where $f\colon M(A) \to B$ is a homomorphism of relational structures and $x$ is an element of $M(A)$. Further, an $n$-ary relation $R$ in $\sigma$ is interpreted as the set of all tuples
\[ (A, f, x_1), \quad \dots \ , \quad (A, f, x_n) \]
such that $R(x_1, \ \dots, \  x_n)$ in $M(A)$.

\begin{example}
    Let $\A$ be the discrete subcategory of graphs, consisting of only the triangle graph, and let $M\colon \A \to \Graph$ be the inclusion of $\A$ into the category of graphs.
    Then, given an arbitrary graph $G$, the graph computed as $\DC_M(G)$ is the disjoint union of $k \times l$ triangles, where $k$ is the number of triangles in $G$ and $l$ is the number of automorphisms of the triangle graph, i.e.\ $l=6$.

\end{example}

\section{The abstract classification theorem}
\label{s:classif-theorem}

In this section we prove Theorem~\ref{t:class-classification}. Since the entire argument can be carried out at the abstract categorical level, we actually prove a general categorical statement that can be applied in other scenarios too. In the following we fix a functor
\[M\colon \A \to \B\]
from a discrete category $\A$ into $\B$. We further assume that the pointwise density comonad $\DC_M$ exists on $\B$ (i.e.\ it is given by the formula~\eqref{e:discrete-formula}).

We start with a useful observation. Recall that an object $C$ is \emph{connected} iff, for every morphism $f\colon C \to \coprod_i A_i$ into a coproduct, the morphism $f$ factors uniquely through one of the inclusion morphisms $\iota_i\colon A_i \to \coprod_i A_i$.\footnote{Equivalently, $C$ is connected iff $\hom(C,-)$ preserves coproducts.
} Whenever a connected $C$ is such that $A \cong C + X$ for some $A,X$, we say that $C$ is a \emph{component} of $A$.

\begin{lemma}\label{l:conn-tech}
    Let $\xi\colon X\to \DC_M(X)$ be a $\DC_M$-coalgebra and let $\iota_C\colon C \to X$ be a component inclusion. Furthermore, let $f\colon M(A) \to X$ be the morphism for which $\xi\circ \iota_C$ decomposes through $\iota_f$, as shown below.
    \begin{align}
        \begin{tikzcd}[ampersand replacement=\&]
            C \rar{\iota_C}\dar[swap]{z} \& X \dar{\xi} \\
            M(A) \ar[dashed,pos=0.40]{ru}{f} \rar{\iota_f} \& \DC_M(X)
        \end{tikzcd}
        \label{eq:conn-tech}
    \end{align}
    Then, also the two triangles in the diagram above commute, that is, $\iota_C = f \circ z$ and $\iota_f = \xi \circ f$.
\end{lemma}
\begin{prf}
The first equality is obtained immediately from the triangle law of coalgebras (cf.\ \eqref{e:coalg}) together with (DC2) from Section~\ref{s:lifting-M} as
\[ \iota_C = \epsilon_C \circ \xi \circ \iota_C = \epsilon_C \circ \iota_f \circ z = f \circ z.\]
To show that also $\iota_f = \xi \circ f$ we apply the square law of coalgebras (cf.\ \eqref{e:coalg}). Observe that
\begin{itemize}
    \item $\delta_X \circ \xi \circ \iota_C = \delta_X \circ \iota_f \circ z = \iota_{\iota_f} \circ z$ by (DC3), and
    \item $\DC_M(\xi) \circ \xi \circ \iota_C = \DC_M(\xi) \circ \iota_f \circ z = \iota_{\xi\circ f} \circ z$ by (DC1).
\end{itemize}
    Because $C$ is connected, the factorisation $C\to M(A) \to \DC_M(\DC_M(C))$ into the coproduct must be unique, hence $\iota_f = \xi \circ f$.
\end{prf}

In the following we need to assume that $\B$ is a \emph{\decomposable{} category}, i.e.\ that
\begin{itemize}
\item every object in $\B$ is (isomorphic to) a coproduct of connected objects, and
\item inclusion morphisms into coproducts $\iota_i\colon a_i \to \coprod_i a_i$ are monomorphisms.
\end{itemize}
We say that an object of $\B$ is \emph{essentially in} $\A$ if it is isomorphic to $M(A)$, for some $A$ in $\A$.

Lemma~\ref{l:conn-tech} directly implies a version of Theorem~\ref{t:class-classification} for connected objects.

\begin{restatable}{lemma}{connectedCoalg}
    \label{l:connected-coalg}
    If $\B$ is a \decomposable{} category, then a connected object of $\B$ is essentially in $\A$ iff it admits a $\DC_M$-coalgebra.
\end{restatable}

\noindent
Since the non-trivial direction is proved similarly to Lemma~\ref{l:connected-components} below, we omit its proof.
Next, we show a useful feature of \decomposable{} categories.

\begin{lemma}
    \label{l:component-factorisation}
    In a \decomposable{} category, a component inclusion $C \to Y$ which factors through a monomorphism $X \to Y$ is a component of $X$ as well.
\end{lemma}
\begin{prf}
    Assume $Y$ is equal to the coproduct $\coprod_i C_i$ and $X$ is equal to $\coprod_j D_j$ for some collections of connected objects $\{C_i\}_i$ and $\{D_j\}_j$. By assumption, the component inclusion $\iota\colon C\to Y$ factors as $m \circ h$ for some monomorphism $m\colon X \to Y$ and a morphism $h\colon C \to X$.

    Since $C$ is connected, $h$ factors through a component inclusion $\iota_j\colon D_j \to X$ as shown in the left diagram below:
    \[
        \begin{tikzcd}
            D_j \dar[swap]{\iota_j} & C \ar[swap]{dl}{h} \dar{\iota} \lar[swap]{h_0} \\
            X \rar{m} & Y
        \end{tikzcd}
        \qquad
        \qquad
        \begin{tikzcd}
            D_j \rar{u}\ar[swap]{dr}{m \circ \iota_j} & C_k\dar{\iota_k} \\
            & Y
        \end{tikzcd}
    \]
    Since $D_j$ is connected, $m \circ \iota_j$ factors through some inclusion $\iota_k\colon C_k \to Y$, as shown in the right diagram above. But then $i = k$ and $u \circ h_0 = \id$ since $\iota_i = m \circ h = m \circ \iota_j \circ h_0 = \iota_k \circ u \circ h_0$ and component inclusions are unique. Furthermore, since both $m$ and $\iota_j$ are monomorphisms by our assumptions, so must be $u$ because $\iota_\ell \circ u = m \circ \iota_j$. Consequently, $u$ is an isomorphism because it is both a monomorphism and a split epimorphism.
\end{prf}

To make progress, we need to assume that $\A$ is \emph{\compbased{}}. This means that, whenever $B\in \B$ is essentially in $\A$ then so is every component of $B$. Note that this condition mirrors the third item in the definition of \compbased{} classes.
With this we show the main technical lemma of this section.

\begin{lemma}\label{l:connected-components}
    If $\A$ is \compbased{} and $\B$ is a \decomposable{} category, then any component $C$ of an object $X$ of $\B$ which admits a $\DC_M$-coalgebra $\xi\colon X\to \DC_M(X)$ is essentially in $\A$.
\end{lemma}
\begin{prf}
    Let $\iota_C\colon C \to X$ be the inclusion morphism of $C$ into $X$. Furthermore, let $f\colon M(A) \to X$ be the morphism such that $\xi\circ \iota_C$ decomposes through $\iota_f\colon M(A) \to \DC_M(X)$ and recall that, by Lemma~\ref{l:conn-tech}, the following diagram commutes.
\begin{align}
    \begin{tikzcd}[ampersand replacement=\&]
        C \rar{\iota_C}\dar[swap]{z} \& X \dar{\xi} \\
        M(A) \rar{\iota_f}\ar[pos=0.40]{ru}{f} \& \DC_M(X)
    \end{tikzcd}
\label{eq:comp-selection}
\end{align}
    Observe that $f$ is a monomorphism since $\iota_f$ is. Therefore, by Lemma~\ref{l:component-factorisation}, $C$ is a component of $M(A)$. Consequently, $C$ is essentially in $\A$ because $\A$ is \compbased{}.
\end{prf}

The main classification theorem, which we state in full, is obtained as a consequence of the previous lemma.
\begin{theorem}\label{t:categ-classification}
    Let $M\colon \A\to \B$ be a functor from a discrete \compbased{} category $\A$ into a \decomposable{} category $\B$ such that the pointwise density comonad $\DC_M$ exists.

    Then an object $b \in \B$ is isomorphic to a coproduct of objects essentially in~$\A$ if and only if $b$ admits a $\DC_M$-coalgebra.
\end{theorem}
\begin{prf}
    The left-to-right implication follows the fact that $M\fl(A)$ is a coalgebra on $M(A)$, for every $A\in \A$ (cf.\ \eqref{eq:M-fl}), and that coalgebras are closed under coproducts that exist in $\B$.
    Conversely, if $\xi\colon X\to \DC_M(X)$ is a coalgebra then, by our assumptions, $X$ is isomorphic to a coproduct $\coprod_i C_i$ of connected objects by Lemma~\ref{l:connected-components} and all those components are essentially in $\A$.
\end{prf}

Observe that both the category of relational structures $\Rs$ and the category of graphs $\Graph$ are \decomposable{} categories.\footnote{Note that for $\Rs$, the connected objects are those structures whose Gaifman graphs are connected.}
Therefore, the previous theorem immediately yields Theorem~\ref{t:class-classification}. Indeed, given a \compbased{} class~$\Delta$ of relational structures or graphs, let $\Delta_C$ be the subclass of $\Delta$ consisting of connected structures only. We then set $\A$ to be a discrete subcategory of $\Rs$ or $\Graph$ consisting of one representative from every isomorphism class in $\Delta_C$. Since we picked only one representative from every equivalence class, the category $\A$ is small. Therefore, the density comonad $\DC_M$, for the inclusion functor $M\colon \A \to \Rs$, exists because both $\Rs$ and $\Graph$ have all (small) coproducts. Observe that the comonad $\DC_M$ classifies $\Delta$. Indeed, by Theorem~\ref{t:categ-classification}, a finite relational structure $B$ has a $\DC_M$-coalgebra if and only if there exist $C_1, \dots, C_n$ in $\Delta_C$ such that $B \cong C_1 + \dots + C_n$. In turn, this is equivalent to $B$ being in $\Delta$, which follows from being \compbased{} as then $C_1 + \dots + C_n$ is in $\Delta$ iff all the individual structures $C_1, \dots, C_n$ are.

\begin{remark}
    In the proof of Theorem~\ref{t:class-classification}, in the previous paragraph, we made sure that all objects in the image of $M$ are connected. This is stronger than assuming that $\A$ is \compbased{}. By carefully inspecting the proof of Theorem~\ref{t:categ-classification} and the preceding lemmas one can check that this extra assumption allows us to drop the requirement that $\B$ is \decomposable{}. See Lemma~\ref{l:connected-coalg-2} below for details.
\end{remark}

\begin{remark}
    Theorem~\ref{t:class-classification} says that, for a class $\Delta$ of structures closed under isomorphisms and finite coproducts, if $\Delta$ is also closed under summands then it can be classified by a comonad. However, the converse does not hold. Let $C$ and $D$ be two graphs with no homomorphism $C \to D$ nor any homomorphisms $D\to C$. This happens, for example, if $C$ is the triangle and $D$ is the cycle on five vertices. Take $\A$ to be the discrete subcategory of $\Graph$ consisting of $C+D$ only and let $M\colon \A \to \Graph$ be the subcategory inclusion. Then, despite $C+D$ admitting a $\DC_M$-coalgebra, no connected graph admits a $\DC_M$-coalgebra (by Lemma~\ref{l:connected-coalg}). It is easy to see that the class of finite structures classified by $\DC_M$ is the class consisting of graphs isomorphic to
    \[ C + \dots + C + D + \dots + D \]
    where both $C$ and $D$ appear in at least one copy in the coproduct. Consequently, the class of structures classified by $\DC_M$ is not closed under summands.
\end{remark}

\subsection{Examples}

The category $\Rs$ of $\sigma$-structures is a \decomposable{} category. Therefore, in our applications we only need to check that a class $\Delta$ of $\sigma$-structures is closed under finite coproducts and summands. These are fairly weak conditions, satisfied by many well-known examples of classes from the literature. In particular, this includes classes of finite structures closed under finite coproducts which are
\begin{enumerate}[label=\arabic*.]
\item monotone, i.e. class closed under taking substructures,
\item hereditary, i.e. class closed under taking induced substructures, or
\item closed under taking graph minors.
\end{enumerate}
Further examples include
\begin{enumerate}[label=\arabic*.]
    \setcounter{enumi}{3}
    \item Fra\"iss\'e classes closed under free amalgamations, or
    \item classes of coproducts of connected cores.
\end{enumerate}
Recall that a core is a structure with the property that all of its endomorphisms are automorphisms. An example of a class from (5) is the class of coproducts of cycles. Note that the discrete density comonad for this class captures co-spectrality, see Section~\ref{s:lovasz}.

As an example of a non-example, take the class of graphs that can be drawn on a surface of genus $n$, for $n>1$. This class is characterised by a finite set of forbidden minors. However, it is not closed under taking coproducts and hence is not a \compbased{} class. On the other hand, any minor-closed class can be completed under finite coproducts. The resulting class is then still minor closed \cite[Lemma 5]{buliandawar2017fixed} and hence is classified by a comonad.\footnote{We are grateful to Anuj Dawar for pointing out these facts.}

\begin{remark}
    The proof of Theorem~\ref{t:class-classification} is carried out abstractly, in the language of category theory, and thus can be dualised. In the dual statement we have monads instead of comonads and instead of \compbased{} classes we have classes closed under isomorphisms,  products, and \emph{factors}, i.e.\ with the property that if $A\times B$ is in the class then so are $A$ and $B$. For such classes there is a monad which classifies the class, i.e.\ a finite structure is in the class iff it admits an algebra for the monad. An example of a class of graphs which can be classified in this way is the class of connected non-bipartite graphs, cf.\ Chapter 8 in~\cite{hammack2011handbook}.
\end{remark}

\section{Graph parameters}
\label{s:graph-param}

A \emph{graph parameter} is a mapping $\mu\colon \GraphFin \to \ol\R$, from the class of finite graphs $\GraphFin$ to the class of extended real numbers $\ol\R = [-\infty,+\infty]$, which gives the same value to any two isomorphic graphs. Moreover, we say that it is \emph{standard}\footnote{Also known as \emph{maxing}, e.g.\ in \cite{lovasz2012networks}.
} if $\mu(G_1 + G_2) = \max\{\mu(G_1),\mu(G_2)\}$.

Standard graph parameters cover many well-known examples of graph parameters from the literature, such as

\begin{itemize}
    \item clique number, chromatic number, max-degree,
    \item tree-depth, tree-width, path-width, clique-width, etc.
\end{itemize}

In this section we show that every standard graph parameter $\mu$ gives rise to a \emph{graded comonad} $(\C_k)_k$, that is, a sequence of comonads $(\C_k)_k$ indexed by extended real numbers and comonad morphisms $g_{k,l}\colon \C_k \nattrans \C_l$, for every $k \leq l$ in $\ol\R$, such that $g_{k,l} = g_{j,l} \circ g_{k,j}$ for any $k\leq j \leq l$ in $\ol\R$.\footnote{In fact, this is a special type of graded comonad, with the grading being over the fixed monoid $(\ol\R,\min,+\infty)$. For details see~\cite{abramskyshah2021relating}.}
Given a graded comonad $(\C_k)_k$ we define the \emph{coalgebra number} $\kappa^\C(G)$, of a graph $G$, to be the infimum of $k\in \ol\R$ such that $G$ admits a $\C_k$-coalgebra \cite{abramskydawarwang2017pebbling,abramskyshah2021relating}.
We show that the coalgebra number for the constructed graded comonad agrees with the standard graph parameter $\mu$ we started with. In other words, we have that $\mu(G) \leq k$ iff $G$ admits a $\C_k$-coalgebra.

Note that every graded comonad $(\C_k)_k$ trivially determines a graph parameter, by setting $\mu(G) := \kappa^\C(G)$. We have already mentioned graded comonads characterising graph parameters this way. For example, the (gra\-ded) Ehrenfeucht–Fra\"iss\'e comonad $(\Ek)_k$ characterises tree-depth~\cite{abramskyshah2021relating}, the pebbling comonad $(\Pk)_k$ characterises tree-width~\cite{abramskydawarwang2017pebbling} and the pebble-relation comonad $(\PRk)_k$ classifies path-width~\cite{montacuteshah2021pebble}.

To start with, observe that there is a one-to-one correspondence between \emph{graph properties}, i.e. graph parameters valued in $\{0,1\}$, and classes of finite graphs which are closed under isomorphisms. Furthermore, it is easy to see that the correspondence restricts to that of standard graph properties and \compbased{} classes of graphs:
\begin{lemma}
    Given a standard graph property $\mu\colon \GraphFin \to \{0,1\}$, the class of graphs $G$ such that $\mu(G) = 0$ is closed under isomorphisms, finite coproducts, and summands. In fact, every such class is obtained from a standard graph property this way.
    \qed
\end{lemma}

Therefore, by Theorem~\ref{t:class-classification}, there is a comonad $\C^\mu$ which classifies the class~$\Delta$ of finite graphs $G$ such that $\mu(G) = 0$, for every standard graph property $\mu$. We construct $\C^\mu$ explicitly, as the pointwise density comonad for the inclusion functor
\begin{align}
    \A \to \Graph,
    \label{eq:param-inclus-func}
\end{align}
where $\A$ is a discrete subcategory of finite connected graphs consisting precisely of one graph from every isomorphism class in $\Delta$. Then, by Theorem~\ref{t:categ-classification}, $\C^\mu$ classifies $\Delta$.

\subsection{Grading graph parameters}
We use this to construct a sequence of comonads for a given standard graph parameter $\mu$. For every extended real number $k$, we turn $\mu$ into a graph property
\[ \mu_{\leq k}\colon \Graph \to \{0,1\} \]
by setting $\mu_{\leq k}(G) = 0$ iff $\mu(G) \leq k$. Then, the density comonad $\C_k^\mu$, defined as~$\C^\mu$ for $\mu := \mu_{\leq k}$, classifies finite graphs $G$ such that $\mu(G) \leq k$.

Moreover, we can make sure that there is a linearly ordered chain of embeddings of discrete categories
\[ \A_{-\infty} \hookrightarrow \dots \hookrightarrow \A_k \hookrightarrow \A_l \hookrightarrow \dots \hookrightarrow \A_{+\infty} \qquad\text{(with $k\leq l$)}\]
where each $\A_k$ is a category as in \eqref{eq:param-inclus-func}, for the class of graphs $G$ such that $\mu(G) \leq k$. Then, by Lemma~\ref{l:comonad-morphism} in Section~\ref{s:composed-functors}, the composite
\[ \A_k \hookrightarrow \A_{l} \to \Graph, \]
for $k \leq l$, gives rise to a comonad morphism $g_{k,l}\colon \C^\mu_k \Rightarrow \C^\mu_{l}$. In fact, we have $g_{k,l} = g_{j,l} \circ g_{k,j}$ for every $k \leq j \leq l$, by functoriality of $\DC_{(-)}$. Hence, $(\C^\mu_k)_k$ is a graded comonad with the property that $\kappa^{\C^\mu}(G) = \mu(G)$ for every finite graph~$G$.

\begin{remark}
The procedure to produce sequences of comonads for standard graph parameters can be defined dually for graph parameters $\mu$ with the property that $-\mu$ is standard, i.e. graph parameters such that $\mu(G_1 + G_2) = \min\{\mu(G_1),\mu(G_2)\}$. This is done by constructing a sequence of standard graph properties
\[ \mu_{\geq k}\colon \Graph \to \{0,1\} \]
and inducing comonads classifying the classes of graphs such that $\mu(G) \geq k$ in a similar spirit as before. This then covers examples of graph parameters such as min-degree and girth.
\end{remark}

\subsection{Comparison with game comonads}
For some graph parameters and classes of structures we already knew how to construct comonads that classify them. In particular, this holds for the classes of structures classified by the comonads $\Pk$, $\Ek$, and $\PRk$. In this section, we explain the relationship between those comonads and discrete density comonads constructed directly for given classes.

In fact, we show that discrete density comonads are weakly initial in the category of comonads that classify the same class. To this end, denote by $\DC_\Delta$ the discrete density comonad constructed as in \eqref{eq:param-inclus-func} above, for a \compbased{} class $\Delta$.

\begin{proposition}
    \label{p:weak-initial}
    Let $\Delta$ be a \compbased{} class of relational structures or graphs and let $\C$ be a comonad that classifies a class $\Gamma$. Then, $\Delta \subseteq \Gamma$ if, and only if, there exists a comonad morphism $\DC_\Delta \nattrans \C$.
\end{proposition}

Observe that the right-to-left direction is immediate as a comonad morphism $\DC_\Delta \nattrans \C$ lifts to a functor $L\colon \EM(\DC_\Delta) \to \EM(\C)$ making the following diagram commute (cf.\ Section~\ref{s:comonad-morph}).
\[
\begin{tikzcd}
    \EM(\DC_\Delta) \ar[swap]{rd}{U^{\DC_\Delta}} \ar{rr}{L} & & \EM(\C) \ar{ld}{U^\C} \\
& \B
\end{tikzcd}
\]
(Here $\B$ is either the category of relational structures or graphs.) For a structure~$B$ in $\Delta$, let $\beta\colon B \to \DC_\Delta(B)$ be a $\DC_\Delta$-coalgebra, which exists because $\DC_\Delta$ classifies~$\Delta$. Then, by the commutativity of the above triangle $L(\beta)$ is a $\C$-coalgebra $B \to \C(B)$ making $B \in \Gamma$ because $\C$ classifies $\Gamma$.

We carry out the left-to-right direction of the proof abstractly, for arbitrary categories rather than just relational structures or graphs. Let $\D := \DC_M$ be the density comonad of a functor $M\colon \A \to \B$ from a discrete category $\A$.
Further, assume that $\C$ is a comonad on $\B$ such that for every $A\in \A$, there exists a coalgebra
\[ \varphi_A\colon M(A)\to \C(M(A)).\]
Observe that, since $\A$ is a discrete category, the collection of morphisms $\{ \varphi_A \mid A\in \A\}$ trivially forms a natural transformation $\varphi\colon M \Rightarrow \C M$. Since $\D$ is a density comonad of $M$, there is a natural transformation $\varphi^*\colon \D \Rightarrow \C$ such that
\begin{align}
    \varphi = \varphi^* M \circ \eta
    \label{e:varphis}
\end{align}
where $\eta\colon M \nattrans \D M$ is the initial natural transformation determining $\D$ (cf.\ Section~\ref{s:density-comonads}). Then, Proposition~\ref{p:weak-initial} follows from the following lemma, which is a direct consequence of
    Theorem II.1.1 in~\cite{dubuc2006kan}. We include its proof in the appendix for completeness.
\begin{restatable}{lemma}{initialMorphism}
    $\varphi^*\colon \D \nattrans \C$ is a morphism of comonads.
\end{restatable}

\begin{example}
    Proposition~\ref{p:weak-initial} gives us that that for our running examples of comonads $\Ek$, $\Pk$, $\PRk$, there exist comonad morphisms
    \[ \DC_{\mathcal{TD}_k} \nattrans \Ek, \quad \DC_{\mathcal{TW}_k} \nattrans \Pk, \qtq{and} \DC_{\mathcal{PW}_k} \nattrans \PRk,\]
    where $\mathcal{TD}_k$, $\mathcal{TW}_k$, and $\mathcal{PW}_k$ are the classes of finite structures of tree-depth, tree-width, and path-width ${\leq}k$, respective.

    Note that unlike  the game comonads $\Ek$, $\Pk$, and $\PRk$, the discrete density comonads $\DC_{\mathcal{TD}_k}$, $\DC_{\mathcal{TW}_k}$, and $\DC_{\mathcal{PW}_k}$ do not classify infinite structures with the corresponding properties.
\end{example}

\subsection{Nowhere dense comonads}
A direct consequence of Proposition~\ref{p:weak-initial} is that we can characterise comonads that classify monotone nowhere dense classes of graphs in terms of non-existence of certain comonad morphisms. Recall that a class $\Delta$ is \emph{somewhere dense} if there exists a natural number $p$ such that, for every $n$, the $p$-th subdivision $K^p_n$ of all edges in the clique graph $K_n$ on $n$ vertices is a subgraph of some graph in $\Delta$.\footnote{A subdivision of a set of edges in a graph replaces each edge in the set by a path of length 2 through a new vertex.} Then, a class is \emph{nowhere dense}  \cite{nevsetvril2011nowhere} if it is not somewhere dense.

It is immediate that a monotone class of graphs $\Delta$ (i.e.\ a class closed under substructures) is somewhere dense if and only if $\Cli_p \subseteq \Delta$, for some $p$, where
\[ \Cli_p = \{ K_n^p \mid n\in \mathbb N \} \]
is the class of $p$-th subdivisions of all cliques. We can now state the characterisation.

\begin{proposition}
    Assume $\C$ classifies a monotone class of graphs $\Delta$. Then, $\Delta$ is nowhere dense if, and only if, there is no comonad morphism $\DClip \nattrans \C$ for any $p\in \mathbb N$.
\end{proposition}
\begin{prf}
    Define $\ol{\Cli_p}$ to be the closure of $\Cli_p$ under finite coproducts. Observe that $\ol{\Cli_p}$ is \compbased{} and, since the connected objects in $\ol{\Cli_p}$ are precisely the objects in $\Cli_p$, the comonad $\DClip$ classifies $\ol{\Cli_p}$. Moreover, since any class classified by a comonad needs to be closed under finite coproducts, $\Cli_p \subseteq \Delta$ iff $\ol{\Cli_p} \subseteq \Delta$. The result follows by monotonicity of $\Delta$ and by Proposition~\ref{p:weak-initial}.
\end{prf}

\section{Lov\'asz-type theorems for free}
\label{s:lovasz}

A classic result of Lov\'{a}sz~\cite{Lovasz1967} says that two finite structures are isomorphic if and only if they admit the same number of homomorphisms from all finite structures. This result has been extended in many different ways. In one type of generalisation, isomorphisms are replaced by a selected equivalence relation $\asymp$ on finite structures, and the class of all finite structures by a class of selected finite structures $\Delta$. Then a typical Lov\'asz-type theorem expresses that, for finite structures $A,B$,
\[ A \asymp B \enspace\iff\enspace \hom(C,A) \cong \hom(C,B) \quad\text{for every } C\in \Delta. \]

A number of well-known equivalence relations on finite structures have been characterised in this way. See Figure~\ref{fig:lovasz-theorems} for an overview of some Lov\'asz-type results.

\begin{figure}[ht]
    \begin{center}
    \begin{tabular}{l|l|l}
        $\Delta$ & $\asymp$ & reference \\
        \hline
        cycles & co-spectrality & (folklore) \\[0.5em]
        trees & fractional isomorphism & \cite{Ramana1994fractional} \\[0.5em]
        bipartite graphs & isomorphic bipartite double covers & \cite{boker2018thesis,lovasz2012networks} \footnotemark \\[0.5em]
        planar graphs & quantum isomorphism & \cite{manvcinskaroberson2020quantum} \\[0.5em]
        tree-depth ${\leq}\, k$ & Duplicator wins the bijective & \cite{grohe2020counting} \\
         & $k$-round \ef{} game & \\[0.5em]
        tree-width ${<}\, k$ & Duplicator wins the bijective  & \cite{dvovrak2010recognizing} \\
         & $k$-pebble game & \\[0.5em]
        path-width ${<}\, k$ & Duplicator wins the bijective & \cite{montacuteshah2021pebble}  \\
         & $k$-pebble relation game & \\[0.5em]
        admitting a $k$-pebble & Duplicator wins the bijective & \cite{dawarjaklreggio2021lov} \\
        tree cover of height ${\leq}\,n$ & $n$-round $k$-pebble E.F.\ game & \\[0.5em]
        synchronization & equivalence in graded modal logic & \cite{dawarjaklreggio2021lov} \\
        trees of height ${\leq}\,k$ & of modal depth ${\leq}\,k$ & \\[0.5em]
        an inner-product & existence of a certain unitary map & \cite{grohe2021homomorphism} \\
        compatible class & between homomorphism tensor spaces &
    \end{tabular}
    \caption{Examples of Lov\'asz-type theorems}
    \label{fig:lovasz-theorems}
    \end{center}
\end{figure}
\footnotetext{
    The \emph{bipartite double cover} of a graph $G$ is the product graph $G \times K_2$ where $K_2$ is the clique on two vertices.

    The fact that isomorphic bipartite double covers correspond to counting homomorphisms from bipartite graphs was worked out by B\"oker in his master thesis \cite{boker2018thesis}. He later observed (in private communication) that the same result already follows from Section 5.4.2 in \cite{lovasz2012networks}.
}

Note that the equivalence relations $\asymp$ in Figure~\ref{fig:lovasz-theorems} corresponding to winning strategies of Duplicator can be equivalently described as logical equivalences with respect to a fragment of first-order logic with counting quantifiers. A comonadic proof of the first two Lov\'asz-type theorems that identify a logic fragment was established in \cite{dawarjaklreggio2021lov} and later adapted in \cite{montacuteshah2021pebble} to obtain the new result for path-width. For the comonadic proof to work it is necessary that the comonad $\C$ classifies the  relation $\asymp$, i.e.\ that $A \asymp B$ holds precisely whenever the cofree coalgebras $F^\C(A)$ and $F^\C(B)$ are isomorphic.

Anuj Dawar has asked (in private communication) whether there are comonads covering the other listed cases as well. In our terminology, this means finding comonads that classify both the class $\Delta$ as well as the corresponding $\asymp$ relation, in the same row of the table. We answer this question in the positive for any \compbased{} class $\Delta$. We show that the discrete density comonad $\C$ that classifies such class also classifies the relation $\asymp$, with $A \asymp B$ whenever $\hom(C,A)\cong \hom(C,B)$ for every $C\in \Delta$, thereby proving Corollary~\ref{c:relation-classification}.

The main ingredient of our proof is the following recent result due to Luca Reggio, proved abstractly for locally finitely presentable categories.

\begin{theorem}[Corollary 5.15 in \cite{reggio2021polyadic}]
    \label{t:luca}
    Let $\C$ be a comonad of finite rank on $\Graph$ or $\Rs$. Then, for two finite structures $A,B$,
    \[ F^\C(A) \cong F^\C(B) \qtq{iff} \hom(C,A) \cong \hom(C,B), \]
    for every finite $C$ which admits a $\C$-coalgebra.
\end{theorem}

We see that in order to prove Corollary~\ref{c:relation-classification}, it is enough to show that the comonad $\C$ constructed in the proof of Theorem~\ref{t:class-classification} has finite rank. In the next section we define finite rank comonads and show that the discrete density comonad constructed in the proof of Theorem~\ref{t:class-classification} does have finite rank.

\subsection{Finite rank comonads}

We work in the general setting of categories, rather than the more concrete setting of relational structures or graphs. Recall that an object $C$ of a category $\B$ is \emph{finitely presentable} \cite{adamekrosicky1994locally} if the functor $\hom(C,-)$ preserves filtered colimits.

Let $\C$ be a comonad over a category $\B$ and let $U\colon \EM(\C) \to \B$ be the usual forgetful functor. We say that a comonad $\C$ has \emph{finite rank} if
\begin{enumerate}
    \item $\C$ is \emph{finitary}, i.e. its underlying functor preserves filtered colimits,
    \item if every morphism of the form $f\colon C\to U(\xi)$, from a finitely presentable $C$, admits a factorisation
    \[ f = C \xrightarrow{f_0} U(\xi_0) \xrightarrow{U(\gamma)} U(\xi) \]
    for some $\gamma\colon Y\to X$ such that $U(\xi_0)$ is finitely presentable, and
    \item this factorisation is essentially unique, i.e. if $g\colon C\to U(\xi_0)$ satisfies $f = U(\gamma) \circ g$ then for some factorisation of $\gamma$ into $\lambda\colon \xi_0\to \xi_0'$ and $\gamma'\colon \xi_0' \to \xi$ such that $U(\xi_0')$ is finitely presentable, $U(\lambda)\circ f_0 = U(\lambda) \circ g$.
\end{enumerate}

Observe that if $U(\gamma)$ in the second item is a monomorphism, then essential uniqueness is automatic. In fact, this is the case in our construction.

In the following we fix a functor
\[ M\colon \A \to \B \]
from a discrete category $\A$ and assume that the pointwise density comonad $\DC_M$ for $M$ exists.

\begin{restatable}{proposition}{finitaryComonad}
    \label{p:finitary-comonad}
    If all objects in the image of $M$ are finitely presentable, then $\DC_M$ is finitary.
\end{restatable}
\begin{proof}[Proof (by courtesy of an anonymous referee).]
    It is enough to check the proof for $M\colon \one \to \B$ where $\one = \{\star\}$ is the discrete category with one object. Indeed, the density comonad for any $\A \to \B$ with $\A$ discrete is computed as a coproduct of density comonads for individual restrictions $\one \to \B$ and coproducts commute with colimits. Then, for an $M\colon \one \to \B$, the density comonad is
    \[ \DC_M(X) = \B(M(\star), X) \cdot M(\star) \]
    where $\cdot$ denotes copower. Consequently, since $M(\star)$ is finitary in $\B$ copowers commute with colimits, for a (small) directed diagram  $D\colon I \to \B$ with a colimit $\colim_i D(i)$,
    \begin{align*}
        \DC_M(\colim_i D(i)) &=     \B(M(\star), \colim_i D(i)) \cdot M(\star) \\
                         &\cong (\colim_i \, \B(M(\star), D(i))) \cdot M(\star) \\
                         &\cong \colim_i \ (\B(M(\star), D(i)) \cdot M(\star)) \cong \colim_i \ \DC_M(D(i)).
        \qedhere
    \end{align*}
\end{proof}

Next, we prove a technical lemma, which is (in some sense) a strengthening of Lemma~\ref{l:connected-components}, and which is needed in the proof of Proposition~\ref{p:finite-rank-conditions} below.

\begin{lemma}
    \label{l:connected-coalg-2}
    Assume that all objects in the image of $M$ are connected. Let \(C\) be a component of \(X\) and let \(\xi\colon X\to \DC_M(X)\) be a $\DC_M$-coalgebra. Then, \(C\) can be equipped with a $\DC_M$-coalgebra \(\gamma\colon C\to \DC_M(C)\) such that the inclusion \(\iota_C\colon C\to X\) is a coalgebra homomorphism $(C, \gamma) \to (X, \xi)$.
\end{lemma}
\begin{prf}
    As in the proof of Lemma~\ref{l:connected-components}, we arrive at the diagram \eqref{eq:comp-selection}.
    This time \(M(A)\) is connected. Therefore, since $f$ is a monomorphism (because so is $\iota_f$), by  Lemma~\ref{l:component-factorisation}, $z$ is an isomorphism. Next we show that $f$ is a coalgebra morphism from $\eta_A\colon M(A) \to \DC_M(M(A))$ to $\xi\colon X \to \DC_M(X)$. To this end, recall that $\eta_A = \iota_g$ for $g = \id\colon M(A) \to M(A)$. Therefore, we obtain the desired $\DC_M(f) \circ \eta_A = \iota_{f \circ \id} = \iota_f = \xi \circ f$ by (DC1).
Consequently, $\iota_C$ is also a coalgebra morphism, for $C$ equipped with the coalgebra structure of $\eta_A$ transported along the isomorphism~$z$.
\end{prf}

Lastly, we show that also conditions (2) and (3) are satisfied for discrete density comonads.

\begin{proposition}
    \label{p:finite-rank-conditions}
    Assume that all objects in the image of $M$ are connected and finitely presentable, and that $\B$ is a \decomposable{} category with finite coproducts.
    Then, every morphism \(f\colon A\to U(\xi)\) in \(\B\), for finitely presentable $A$ in $\B$, admits a unique factorisation (up to isomorphism)
\[ f \ = \ A \xrightarrow{ \ g \ } U(\xi_0) \xrightarrow{ \ U(\gamma) \ } U(\xi) \]
where \(U(\xi_0)\) is finitely presentable.
\end{proposition}
\begin{prf}
    In the following, we denote by \(U\) the forgetful functor $U^\C\colon \EM(\DC_M) \to \B$.
    By Theorem~\ref{t:categ-classification}, we may assume that the underlying object \(X\) of \(\xi\colon X\to \DC_M(X)\) is a coproduct \(\coprod_{i\in I} C_i\) of a collection of connected, finitely presentable objects~\(C_i\) essentially in \(\A\).

Recall that \(\coprod_{i\in I} C_i\) is isomorphic to the directed colimit of the following directed diagram
\[ \{ \coprod_{i\in F} C_i \mid F \text{ is a finite subset of } I\} \]
with the obvious morphisms between these finite coproducts. Since \(A\) is finitely presentable, \(f\colon A \to X\) decomposes as
\[ f = A \xrightarrow{g} \coprod_{i\in F} C_i \xrightarrow{\iota_F} X \]
for some finite \(F\subseteq I\).
    By Lemma~\ref{l:connected-coalg-2}, for each \(i\in F\),
    the inclusion morphism \(\iota_i\colon C_i \to X\) is a coalgebra morphism \((C_i, \xi_i) \to (X,\xi)\), for some comonad coalgebra \(\xi_i\colon C_i \to \DC_M(C_i)\). 

    Lastly, because the forgetful functor \(U\colon \EM(\DC_M) \to \B\) creates colimits (see e.g.\ Proposition 20.12 in~\cite{AHS1990}), \(\coprod_{i\in F} C_i\) can be equipped with the coalgebra structure of the coproduct of the coalgebras \(\xi_i\colon C_i \to \DC_M(C_j)\). Moreover, the morphism $\iota_F\colon \coprod_{i\in F} C_i \to X$ is a coalgebra morphism because each of its components is. Also, \(\coprod_{i\in F} C_j\) is finitely presentable because it is a finite coproduct of finitely presentable objects (see e.g.\ Proposition 1.3 in~\cite{adamekrosicky1994locally}). Finally, \(g\) is unique because \(\B\) is a \decomposable{} category and \(\iota_F\) is the inclusion morphism of $\coprod_F C_i$ into the coproduct $\coprod_I C_i$ and hence a monomorphism.
\end{prf}

As a corollary of Propositions~\ref{p:finitary-comonad} and \ref{p:finite-rank-conditions} we obtain the main theorem of this section.

\begin{theorem}
    \label{t:finite-rank}
    Let $M\colon \A \to \B$ be a functor from a discrete category $\A$ to a \decomposable{} category $\B$ with finite coproducts and assume that the pointwise density comonad $\DC_M$ for $M$ exists. If all objects in the image of $M$ are connected and finitely presentable then $\DC_M$ has finite rank.
\end{theorem}

Observe that the constructed functor $M\colon \A \to \B$ in the proof of Theorem~\ref{t:class-classification} (cf.\ the paragraph following Theorem~\ref{t:categ-classification}) automatically satisfies the assumptions of Theorem~\ref{t:finite-rank}. Indeed, $M$ is an inclusion of a class of finite connected structures and finite structures are precisely the finitely presentable objects in the category of relational structures or the category of graphs. Therefore, Theorem~\ref{t:finite-rank} together with Theorem~\ref{t:luca} concludes the proof of Corollary~\ref{c:relation-classification}.

\section{Conclusion}

In this paper we have shown that classes of structures closed under isomorphism, disjoint unions, and summands can always be classified by a comonad, and moreover this comonad admits a Lov\'asz-type theorem. We have also shown that standard graph parameters give rise to graded comonads, i.e.\ sequences of comonads indexed by real numbers, such that the graph parameter is captured by the coalgebra number of a given structure. Both results cover a huge range of examples of structure classes and graph parameters from the literature.

The comonads we construct are, in some sense, the minimal solutions to this problem in that they are weakly initial among the comonads that classify the same class of structures. Our proofs show is that the classifying comonads (or graded comonads) can be very simple and do not need to be specifically tailored for the concept at hand.

Conversely, 
however, the power of game comonads is that they shed light on previously known constructions and reveal new connections between them. 
This can lead to new results. For example, the links between game comonads, logic fragments and combinatorial parameters established in \cite{abramskydawarwang2017pebbling,abramskyshah2021relating} are leveraged in  \cite{jaklmarsdenshah2021mso} and  \cite{montacuteshah2021pebble} to obtain new results for other combinatorial properties simply by changing the comonad at hand.

In \cite{abramsky2021arboreal} the common structure exhibited by game comonads is axiomatised in terms of \emph{arboreal categories} and \emph{arboreal covers}.
This suggests one line of further development, by relating the general results using discrete density comonads of the present paper to the axiomatic setting of \cite{abramsky2021arboreal}.
This can provide a basis for general transfer results of this kind between comonads arising from arboreal covers.

Another potential source of useful comonads is when a particular class of structures is given by a \emph{construction}, similar to the inductive definition of clique-width or the algebraic definition of planar graphs found e.g.\ in \cite{manvcinskaroberson2020quantum}. We hope to explore comonads arising from inductive constructions in  future work.

\section*{Acknowledgements}
We thank the referees for careful reading and many useful suggestions for improvements.
SA thanks Miko\l{}aj Boja\'nczyk and Bartek Klin for very stimulating and enjoyable discussions in Warsaw in 2018 on finding a comonad to classify clique-width, which led to some initial results.
TJ thanks Anuj Dawar who asked for a comonad for co-spectrality, which was another stimulus for this research.
We  thank all the members
of the EPSRC project on Resources and Coresources for their encouragement and feedback.

\bibliographystyle{amsplain}
\bibliography{refs}

\providecommand{\bysame}{\leavevmode\hbox to3em{\hrulefill}\thinspace}
\providecommand{\MR}{\relax\ifhmode\unskip\space\fi MR }
\providecommand{\MRhref}[2]{%
  \href{http://www.ams.org/mathscinet-getitem?mr=#1}{#2}
}
\providecommand{\href}[2]{#2}
\begin{thebibliography}{10}

\bibitem{abramskydawarwang2017pebbling}
Samson Abramsky, Anuj Dawar, and Pengming Wang, \emph{The pebbling comonad in
  finite model theory}, Proceedings of the 32nd Annual ACM/IEEE Symposium on
  Logic in Computer Science (LICS), IEEE, 2017, pp.~1--12.

\bibitem{abramskymarsden2021comonadic}
Samson Abramsky and Dan Marsden, \emph{Comonadic semantics for guarded
  fragments}, Proceedings of the 36th Annual ACM/IEEE Symposium on Logic in
  Computer Science (LICS), IEEE, 2021, pp.~1--13.

\bibitem{abramsky2021arboreal}
Samson Abramsky and Luca Reggio, \emph{Arboreal categories and resources}, 48th
  International Colloquium on Automata, Languages, and Programming (ICALP
  2021), Schloss Dagstuhl-Leibniz-Zentrum f{\"u}r Informatik, 2021,
  pp.~115:1--115:20.

\bibitem{abramskyshah2021relating}
Samson Abramsky and Nihil Shah, \emph{Relating structure and power: Comonadic
  semantics for computational resources}, Journal of Logic and Computation
  \textbf{31} (2021), no.~6, 1390--1428.

\bibitem{abramskytzevelekos2010introduction}
Samson Abramsky and Nikos Tzevelekos, \emph{Introduction to categories and
  categorical logic}, New structures for physics, Springer, 2010, pp.~3--94.

\bibitem{AHS1990}
J.~Ad\'{a}mek, H.~Herrlich, and G.~E. Strecker, \emph{Abstract and concrete
  categories}, John Wiley \& Sons, Inc., New York, 1990.

\bibitem{adamekrosicky1994locally}
Ji{\v{r}}{\'i} Ad{\'a}mek and Ji\v{r}{\'i} Rosick{\'y}, \emph{Locally
  presentable and accessible categories}, vol. 189, Cambridge University Press,
  1994.

\bibitem{appelgate1969categories}
H.~Appelgate and M.~Tierney, \emph{Categories with models}, Seminar on Triples
  and Categorical Homology Theory, Lecture Notes in Mathematics, vol.~80,
  Springer, 1969, pp.~156--244.

\bibitem{awodey2010category}
Steve Awodey, \emph{Category theory}, Oxford University Press, 2010.

\bibitem{boker2018thesis}
Jan B\"oker, \emph{Structural similarity and homomorphism counts}, Master's
  thesis, RWTH Aachen, 2018.

\bibitem{buliandawar2017fixed}
Jannis Bulian and Anuj Dawar, \emph{Fixed-parameter tractable distances to
  sparse graph classes}, Algorithmica \textbf{79} (2017), no.~1, 139--158.

\bibitem{conghailedawar2021game}
Adam~{\'O} Conghaile and Anuj Dawar, \emph{Game comonads \& generalised
  quantifiers}, 29th EACSL Annual Conference on Computer Science Logic (CSL
  2021) (Dagstuhl, Germany) (Christel Baier and Jean Goubault-Larrecq, eds.),
  Leibniz International Proceedings in Informatics (LIPIcs), vol. 183, Schloss
  Dagstuhl--Leibniz-Zentrum f{\"u}r Informatik, 2021, pp.~16:1--16:17.

\bibitem{dawarjaklreggio2021lov}
Anuj Dawar, Tom\'{a}\v{s} Jakl, and Luca Reggio, \emph{Lov\'{a}sz-type theorems
  and game comonads}, Proceedings of the 36th Annual ACM/IEEE Symposium on
  Logic in Computer Science (LICS) (New York, NY, USA), Association for
  Computing Machinery, 2021, pp.~1--13.

\bibitem{diers1986categories}
Yves Diers, \emph{Categories of {B}oolean sheaves of {S}imple {A}lgebras},
  Categories of {B}oolean Sheaves of Simple Algebras, Springer, 1986,
  pp.~48--113.

\bibitem{dubuc2006kan}
Eduardo.~J. Dubuc, \emph{Kan extensions in enriched category theory}, Lecture
  Notes in Mathematics, vol. 145, Springer, 1970.

\bibitem{dvovrak2010recognizing}
Zden{\v{e}}k Dvo{\v{r}}{\'a}k, \emph{On recognizing graphs by numbers of
  homomorphisms}, Journal of Graph Theory \textbf{64} (2010), no.~4, 330--342.

\bibitem{grohe2020counting}
Martin Grohe, \emph{Counting bounded tree depth homomorphisms}, Proceedings of
  the 35th Annual ACM/IEEE Symposium on Logic in Computer Science (LICS), 2020,
  pp.~507--520.

\bibitem{grohe2021homomorphism}
Martin Grohe, Gaurav Rattan, and Tim Seppelt, \emph{Homomorphism tensors and
  linear equations}, accepted at the 49th EATCS International Colloquium on
  Automata, Languages and Programming (ICALP), 2022.

\bibitem{hammack2011handbook}
Richard~H. Hammack, Wilfried Imrich, and Sandi Klav{\v{z}}ar, \emph{Handbook of
  product graphs}, vol.~2, CRC Press, Taylor \& Francis Group, 2011.

\bibitem{jaklmarsdenshah2021mso}
Tom{\'a}\v{s} Jakl, Dan Marsden, and Nihil Shah, \emph{A game comonadic account
  of {C}ourcelle and {F}eferman-{V}aught-{M}ostowski theorems}, In preparation,
  arXiv preprint, 2022.

\bibitem{kock1966continuous}
Anders Kock, \emph{Continuous {Y}oneda representation of a small category},
  {U}niversity of {A}arhus, Denmark, 1966.

\bibitem{leinster2013codensity}
Tom Leinster, \emph{Codensity and the ultrafilter monad}, Theory and
  Applications of Categories \textbf{28} (2013), no.~13, 332--370.

\bibitem{Lovasz1967}
L{\'a}szl{\'o} Lov\'{a}sz, \emph{Operations with structures}, Acta Math. Acad.
  Sci. Hungar. \textbf{18} (1967), 321--328.

\bibitem{lovasz2012networks}
L{\'a}szl{\'o} Lov{\'a}sz, \emph{Large networks and graph limits}, vol.~60,
  American Mathematical Soc., 2012.

\bibitem{manvcinskaroberson2020quantum}
Laura Man{\v{c}}inska and David~E. Roberson, \emph{Quantum isomorphism is
  equivalent to equality of homomorphism counts from planar graphs}, 2020 IEEE
  61st Annual Symposium on Foundations of Computer Science (FOCS), IEEE, 2020,
  pp.~661--672.

\bibitem{montacuteshah2021pebble}
Yo{\`a}v Montacute and Nihil Shah, \emph{The {P}ebble-{R}elation {C}omonad in
  {F}inite {M}odel {T}heory}, accepted at the 37th Annual ACM/IEEE Symposium on
  Logic in Computer Science (LICS), 2021.

\bibitem{nevsetvril2011nowhere}
Jaroslav Ne{\v{s}}et{\v{r}}il and Patrice~Ossona De~Mendez, \emph{On nowhere
  dense graphs}, European Journal of Combinatorics \textbf{32} (2011), no.~4,
  600--617.

\bibitem{Ramana1994fractional}
Motakuri~V. Ramana, Edward~R. Scheinerman, and Daniel Ullman, \emph{Fractional
  isomorphism of graphs}, Discrete Mathematics \textbf{132} (1994), 247--265.

\bibitem{reggio2021polyadic}
Luca Reggio, \emph{Polyadic sets and homomorphism counting}, Submitted, arXiv
  preprint, 2021.

\bibitem{wisbauer2008algebras}
Robert Wisbauer, \emph{Algebras {V}ersus {C}oalgebras}, Applied Categorical
  Structures \textbf{16} (2008), no.~1, 255--295.

\end{thebibliography}

\appendix

\section{Omitted Proofs}

\connectedCoalg*
\begin{prf}
    Let $C\in \B$ be connected.
The left-to-right implication follows immediately from the fact that, $M\fl(A)$ is a coalgebra on $M(A)$, for every $A\in \A$ (cf.\ \eqref{eq:M-fl}). For the right-to-left implication, let $\beta\colon C\to \DC_M(C)$ be a coalgebra. Since $C$ is connected,
\[ \beta\colon C \to \coprod_{A\in \A} \ \coprod_{f\colon M(A) \to C} M(A) \]
    uniquely factors as $\beta = \iota_f \circ \beta_0$, for some $\beta_0\colon C \to M(A)$ and $f\colon M(A) \to C$. By Lemma~\ref{l:conn-tech}, the following diagram commutes.
    \[
        \begin{tikzcd}
            C \rar{\id}\dar[swap]{\beta_0} & C \dar{\beta} \\
            M(A) \rar{\iota_f}\ar[pos=0.40]{ru}{f} & \DC_M(C)
        \end{tikzcd}
    \]
    One triangle immediately gives $f \circ \beta_0 = \id$ and the other triangle together with $\beta = \iota_f \circ \beta_0$ and the fact that inclusions are monomorphisms (in \decomposable{} categories) entails $\beta_0 \circ f = \id$.
\end{prf}

\initialMorphism*
\begin{prf}
First, let us check commutativity of the triangle diagram in \eqref{e:comonad-morphisms}. By \eqref{eq:eta-comonad}, $\epsilon^\D M \circ \eta = \id$. Similarly, by \eqref{e:varphis}, $\epsilon^\C \circ \varphi^* \circ \eta = \epsilon^\C \circ \varphi = \id$ where the last equality follows from commutativity of the triangle law in \eqref{e:coalg} for every individual coalgebra $\varphi_A\colon M(A) \to \C(M(A))$. We see that $\epsilon^\D M \circ \eta = (\epsilon^\C \circ \varphi^*) M \circ \eta$ from which it follows that $\epsilon^\D = \epsilon^\C \circ \varphi^*$ by the universal property of $\eta$.

Next, we check the oblong law in \eqref{e:comonad-morphisms}. The composition $\delta^\C M \circ \varphi^* M \circ \eta$ is equal to $\delta^\C M \circ \varphi$, by \eqref{e:varphis}, which is equal to $\C \varphi \circ \varphi$ by the square law of coalgebras \eqref{e:coalg}. Similarly,
\begin{align*}
\varphi^* \C M \circ \D \varphi^* M \circ \delta^\D M \circ \eta
&= \varphi^* \C M \circ \D \varphi^* M \circ \D \eta \circ \eta \\
&= \varphi^* \C M \circ \D \varphi \circ \eta \\
&= \C \varphi \circ \varphi^* M \circ \eta  \\
&= \C \varphi \circ \varphi
\end{align*}
The first equality follows from \eqref{eq:eta-comonad}, the second and fourth from \eqref{e:varphis} and the third is naturality of $\varphi^*$. We have shown that $(\varphi^* \C \circ \D \varphi^* \circ \delta^\D) M \circ \eta$ is equal to $(\delta^\C \circ \varphi^*) M \circ \eta$ which, by the universal property of $\eta$, implies that $\varphi^* \C \circ \D \varphi^* \circ \delta^\D = \delta^\C \circ \varphi^*$.
\end{prf}

\end{document}